\documentclass[a4paper,10pt,bibtotoc,tocindent,fleqn]{article}

\usepackage[english]{babel}
\usepackage{amsmath}
\usepackage{amsthm}
\usepackage{amssymb}
\usepackage{amsfonts}
\usepackage{amsxtra}
\usepackage{color}
\usepackage[breaklinks]{hyperref}
\usepackage{tocloft}
\usepackage{makeidx}
\usepackage{textcomp}
\usepackage{rotating}
\usepackage{multirow}
\usepackage{booktabs,longtable}
\usepackage[all]{xy}
\usepackage{float}
\usepackage{rotfloat}
\usepackage{paralist}

\numberwithin{equation}{section}
\newcommand{\td}{\,\mathrm{d}}

\renewcommand\Re{\operatorname{Re}}
\renewcommand\Im{\operatorname{Im}}

\DeclareMathOperator{\tr}{tr}
\DeclareMathOperator{\rk}{rk}

\newcommand{\Sym}{\operatorname{Sym}}

\newcommand{\Herm}{\operatorname{Herm}}
\newcommand{\Str}{\operatorname{Str}}

\newcommand{\Aut}{\operatorname{Aut}}

\newcommand{\GL}{\operatorname{GL}}

\newcommand{\PSL}{\operatorname{PSL}}
\renewcommand{\sl}{\mathfrak{sl}}

\renewcommand{\sp}{\mathfrak{sp}}

\newcommand{\SO}{\operatorname{SO}}
\newcommand{\so}{\mathfrak{so}}

\newcommand{\PSU}{\operatorname{PSU}}
\newcommand{\su}{\mathfrak{su}}

\newcommand{\End}{\operatorname{End}}
\newcommand{\RR}{\mathbb{R}}
\newcommand{\KK}{\mathbb{K}}
\newcommand{\CC}{\mathbb{C}}

\newcommand{\NN}{\mathbb{N}}

\newcommand{\HH}{\mathbb{H}}
\newcommand{\OO}{\mathbb{O}}

\newcommand{\BB}{\mathbb{B}}

\newcommand{\FF}{\mathbb{F}}

\renewcommand{\1}{\mathbf{1}}

\newcommand{\calE}{\mathcal{E}}
\newcommand{\calF}{\mathcal{F}}
\newcommand{\calO}{\mathcal{O}}

\newcommand{\calB}{\mathcal{B}}

\newcommand{\calP}{\mathcal{P}}
\newcommand{\calU}{\mathcal{U}}

\newcommand{\calD}{\mathcal{D}}

\newcommand{\calR}{\mathcal{R}}
\newcommand{\calI}{\mathcal{I}}
\newcommand{\calJ}{\mathcal{J}}
\newcommand{\calK}{\mathcal{K}}

\newcommand{\frakg}{\mathfrak{g}}

\newcommand{\frakk}{\mathfrak{k}}
\newcommand{\frakp}{\mathfrak{p}}

\DeclareMathOperator{\Det}{Det}

\newcommand{\diag}{\operatorname{diag}}

\newcommand{\even}{\textup{even}}

\theoremstyle{plain}
\newtheorem{theorem}{Theorem}[section]
\newtheorem{proposition}[theorem]{Proposition}
\newtheorem{lemma}[theorem]{Lemma}

\newtheorem{thmalph}{Theorem}

\theoremstyle{definition}

\newtheorem{example}[theorem]{Example}

\newtheorem{remark}[theorem]{Remark}

\theoremstyle{remark}
\renewenvironment{pf}{\begin{proof}[\sc Proof]}{\end{proof}}

\numberwithin{equation}{section}

\begin{document}

\title{Heat kernel analysis for Bessel operators on symmetric cones}
\author{Jan M\"ollers}
\date{August 31, 2012}
\maketitle
\begin{abstract}

We investigate the heat equation corresponding to the Bessel operators on a symmetric cone $\Omega=G/K$. These operators form a one-parameter family of elliptic self-adjoint second order differential operators and occur in the Lie algebra action of certain unitary highest weight representations. The heat kernel is explicitly given in terms of a multivariable $I$-Bessel function on $\Omega$. Its corresponding heat kernel transform defines a continuous linear operator between $L^p$-spaces. The unitary image of the $L^2$-space under the heat kernel transform is characterized as a weighted Bergmann space on the complexification $G_\CC/K_\CC$ of $\Omega$, the weight being expressed explicitly in terms of a multivariable $K$-Bessel function on $\Omega$. Even in the special case of the symmetric cone $\Omega=\RR_+$ these results seem to be new.\\



\end{abstract} 

%\newpage

%\tableofcontents

%\newpage

\addcontentsline{toc}{section}{Introduction}
\section*{Introduction}

Heat equations are important parabolic partial differential equations and have been studied in various different settings. Classically one considers the heat equation corresponding to the Laplace--Beltrami operator $\Delta$ on a Riemannian manifold $X$:
\begin{align}
 \Delta_xu(t,x) &= \partial_tu(t,x), & \mbox{on $(0,\infty)\times X$.}\label{eq:ClassicalHeatEquation}
\end{align}
The properties of $\Delta$ crucial for the heat equation are:
\begin{itemize}
\item $\Delta$ is an elliptic second-order differential operator on $X$,
\item It extends to a self-adjoint operator on $L^2(X)$ with spectrum contained in $(-\infty,0]$.
\end{itemize}
The corresponding heat kernel transform $e^{t\Delta}$ is well-defined by the functional calculus and gives a smoothing operator on $L^2(X)$ for $t>0$ constructing solutions to the heat equation \eqref{eq:ClassicalHeatEquation} with given initial values. A classical problem in harmonic analysis is to determine the image of $L^2(X)$ under the heat kernel transform. For $X=\RR^n$ Bargmann and Segal noticed that every function in the image of $e^{t\Delta}$ has an analytic extension to the complexification $X_\CC=\CC^n$. In fact, $e^{t\Delta}$ is a unitary isomorphism from $L^2(\RR^n)$ onto the weighted Bergmann space of holomorphic functions on $\CC^n$ which are square-integrable with respect to the measure $e^{\frac{1}{2t}|\Im z|^2}\td z$.

These observations have a natural generalization to Riemannian symmetric spaces $X$. For the case of compact Riemannian symmetric spaces this was first established by Hall \cite{Hal94}. If $X=U/K$ with $U$ a compact Lie group and $K$ a symmetric subgroup then the heat kernel transform $e^{t\Delta}$ extends to a unitary operator between $L^2(U/K)$ and a weighted Bergmann space on the complexification $X_\CC=U_\CC/K_\CC$. An alternative proof for the unitarity of the transform was given by Hilgert--Zhang \cite{HZ09} using the so-called restriction principle which was introduced by {\'O}lafsson--{\O}rsted \cite{OO96}. For semisimple Riemannian symmetric spaces of non-compact type Kr{\"o}tz--\'Olafsson--Stanton \cite{KOS05} characterized the unitary image of the heat kernel transform. Let $X=G/K$ with $G$ a non-compact semisimple Lie group and $K$ a maximal compact subgroup. The main difference to the compact case is that functions in the image of the heat kernel transform do in general not extend to the whole complexification $X_\CC=G_\CC/K_\CC$, but only to a $G$-invariant subdomain in $X_\CC$, the complex crown (also referred to as the Akhiezer--Gindikin domain). Further, it is shown in \cite{KOS05} that in general the unitary image of $L^2(X)$ under the heat kernel transform cannot be a weighted Bergmann space on the complex crown.\\

Here we restrict our attention to symmetric cones $\Omega=G/K$, a certain class of non-compact Riemannian symmetric spaces with $G$ is a connected reductive Lie group with one-dimensional center and $K\subseteq G$ a maximal compact subgroup. Classical examples are $\Omega=\RR_+$, the Lorentz cone $\Omega=\{x\in\RR^n:x_1>\sqrt{x_2^2+\cdots+x_n^2}>0\}$ or the cone of positive definite Hermitian $n\times n$ matrices over $\RR$, $\CC$ or $\HH$. Explicit expressions for the classical heat kernel in these examples are given in \cite{Din99,Saw06}. In this paper, however, we consider not the Laplacian on $\Omega$, but a certain family of differential operators $B_\lambda$ which arise from unitary representation theory (see Section \ref{sec:BesselOperators} for the precise definition). In fact, for $\lambda>c(\Omega)$ (the constant $c(\Omega)>0$ depending only on the cone $\Omega$, see \eqref{eq:DefcOmega}) the so-called \textit{Bessel operators} $B_\lambda$ occur in the Lie algebra action of certain unitary representations, namely $L^2$-models for the analytic continuation of the holomorphic discrete series corresponding to the automorphism group of the associated tube domain $V+i\Omega$. For this we realize $\Omega$ as an open convex cone in an ambient Euclidean vector space $V$ with inner product $(-|-)$. As in the classical setting the operators $B_\lambda$ satisfy the following properties:
\begin{itemize}
\item Each $B_\lambda$ is an elliptic second-order differential operator on $\Omega$,
\item For $\lambda>c(\Omega)$ the operator $B_\lambda$ extends to a self-adjoint operator on $L^2(\Omega,\td\mu_\lambda)$ with spectrum given by $(-\infty,0]$.
\end{itemize}
Here $\td\mu_\lambda$ is a certain $G$-equivariant measure on $\Omega$ (see Section \ref{sec:EquivariantMeasures}). We remark that in contrast to the $G$-invariant Laplacian on $\Omega$ the operators $B_\lambda$ are merely $K$-invariant. Let $H_\lambda(t):=e^{tB_\lambda}$ be the corresponding heat kernel transform on $L^2(\Omega,\td\mu_\lambda)$. The operators $H_\lambda(t)$ are smoothing, $K$-equivariant and satisfy the semigroup property $H_\lambda(s)\circ H_\lambda(t)=H_\lambda(s+t)$.

To describe the heat kernel we make use of the $I$-Bessel function on the symmetric cone $\Omega$ introduced in \cite{Moe12}. It can be constructed using the canonical structure of a Jordan algebra on $V$ (see Section \ref{sec:BesselFunctions} for details). The $I$-Bessel function of parameter $\lambda>c(\Omega)$ on $\Omega$ is an analytic function $\calI_\lambda(x,y)$ on $\Omega\times\Omega$ solving the differential equation
\begin{align*}
 (B_\lambda)_x\calI_\lambda(x,y) &= \tr(y)\calI_\lambda(x,y), & x,y\in\Omega,
\end{align*}
where $\tr(y)=(y|e)>0$, $y\in\Omega$, denotes the trace form, $e\in\Omega$ being the fixed point of $K$. (In fact $e$ is the unit element of the Jordan algebra $V$ and $\tr(y)$ the Jordan trace.) Put
\begin{align*}
 h_\lambda(t,x,y) &:= (2t)^{-r\lambda}e^{-\frac{1}{t}(\tr(x)+\tr(y))}\calI_\lambda\left(\frac{x}{t},\frac{y}{t}\right), & t>0,x,y\in\Omega,
\end{align*}
where $r=\rk\Omega$ is the rank of the symmetric space $G/K$.

\begin{thmalph}[see Theorems \ref{thm:HeatKernelTransform} \&\ \ref{thm:HeatKernelSemigroup}]
For $\lambda>c(\Omega)$ the heat kernel transform $H_\lambda(t)$ on $L^2(\Omega,\td\mu_\lambda)$ is the integral operator given by
\begin{align*}
 H_\lambda(t)f(x) &= \int_\Omega{h_\lambda(t,x,y)f(y)\td\mu_\lambda(y)}, & x\in\Omega.
\end{align*}
For any $1\leq p\leq\infty$ the integral converges absolutely for $f\in L^p(\Omega,\td\mu_\lambda)$ and defines a continuous linear operator on $L^p(\Omega,\td\mu_\lambda)$.
\end{thmalph}

This explicit description of the heat kernel allows to describe the image of $L^2(\Omega,\td\mu_\lambda)$ under the heat kernel transform as a weighted Bergmann space on the natural complexification $\Xi:=G_\CC/K_\CC$ of $\Omega$. The symmetric space $\Xi$ can be realized as an open dense cone in the complexification $V_\CC$ of $V$ which contains $\Omega\subseteq V$ as a totally real submanifold. Since $\Xi$ is dense in $V_\CC$ every entire function on $V_\CC$ is uniquely determined by its values on $\Xi$ and we let $\calO(\overline{\Xi})$ be the space of holomorphic functions on $\Xi$ which extend to $V_\CC$. From the measures $\td\mu_\lambda$ on $\Omega$ one constructs certain $G_\CC$-equivariant measures $\td\nu_\lambda$ on $\Xi$ (see \eqref{eq:EquivMeasuresOnXi}). To describe the weight for the Bergmann space we make use of the $K$-Bessel function $\calK_\lambda(x)$ of parameter $\lambda$ on the symmetric cone $\Omega$. This function was first introduced by Clerc \cite{Cle88} and studied further in \cite{Moe12}. We construct a density $\omega_\lambda(z)$ on $\Xi$ in terms of the $K$-Bessel function which solves the differential equation
\begin{align*}
 B_\lambda\omega_\lambda(z) &= \frac{1}{4}\overline{\tr(z)}\omega_\lambda(z), & z\in\Xi,
\end{align*}
where $\tr(z)$ denotes the $\CC$-linear extension of the trace to $V_\CC$. We then define a weighted Bergmann space $\calF_{\lambda,t}(\Xi)$ on $\Xi$ by
\begin{align*}
 \calF_{\lambda,t}(\Xi) &:= \left\{F\in\calO(\overline{\Xi}):\int_\Xi{|F(z)|^2e^{\frac{1}{t}\Re\tr(z)}\omega_\lambda\left(\frac{z}{t}\right)\td\nu_\lambda(z)}<\infty\right\}
\end{align*}
and endow it with the obvious $L^2$-inner product (suitably normalized, see \eqref{eq:BergmannSpaceInnerProduct} for details). Then for $\lambda>c(\Omega)$ the space $\calF_{\lambda,t}(\Xi)$ turns into a Hilbert space of holomorphic functions with continuous point evaluations. Its reproducing kernel is the analytic continuation of the heat kernel $h_\lambda(2t,z,w)$ of parameter $2t$ with $z,w\in\Xi$, holomorphic in $z$ and antiholomorphic in $w$ (see Theorem \ref{thm:HeatKernelTrafoImage}).

\begin{thmalph}[see Proposition \ref{prop:HeatKernelTrafoAnalyticExtension} and Theorem \ref{thm:HeatKernelTrafoImage}]
Let $\lambda>c(\Omega)$, then for every $f\in L^2(\Omega,\td\mu_\lambda)$ the function $H_\lambda(t)f$ on $\Omega$ extends uniquely to a holomorphic function $\widetilde{H}_\lambda(t)f\in\calF_{\lambda,t}(\Xi)$. The operator $\widetilde{H}_\lambda(t):L^2(\Omega,\td\mu_\lambda)\to\calF_{\lambda,t}(\Xi)$ is a unitary isomorphism.
\end{thmalph}

Our proof uses the intertwining operator between two different realizations of certain unitary highest weight representations (see \cite{HKMO12,Moe12}). However, there is an alternative proof for the unitarity of the heat kernel transform $\widetilde{H}_\lambda(t)$ using the restriction principle. For this one considers the unbounded operator $\calR_{\lambda,t}:\calD(\calF_{\lambda,t}(\Xi))\to L^2(\Omega,\td\mu_\lambda)$ given by restriction from $\Xi$ to $\Omega$. The heat kernel transform $\widetilde{H}_{\lambda,t}$ can then be identified with the unitary part $\calU_{\lambda,t}$ in the polar decomposition $\calR_{\lambda,t}^*=\calU_{\lambda,t}\circ|\calR_{\lambda,t}|$ of the adjoint operator (see Section \ref{sec:RestrictionPrinciple} for details).\\

We briefly illustrate our results for the cone $\Omega=\RR_+$ with equivariant measures $\td\mu_\lambda(x)=x^{\lambda-1}\td x$, $\lambda>c(\Omega)=0$ (see Section \ref{sec:Example} for details). In this case the Bessel operator $B_\lambda$ is given by
\begin{align*}
 B_\lambda &= x\frac{\td^2}{\td x^2}+\lambda\frac{\td}{\td x},
\end{align*}
and the heat kernel $h_\lambda(t,x,y)$ and the density $\omega_\lambda(z)$ take the form
\begin{align*}
 h_\lambda(t,x,y) &= (2t)^{-\lambda}e^{-\frac{x+y}{t}}\widetilde{I}_{\lambda-1}\left(2\frac{\sqrt{xy}}{t}\right), && t>0,x,y\in\RR_+,\\
 \omega_\lambda(z) &= 2\widetilde{K}_{\lambda-1}(|z|), && z\in\CC^\times,
\end{align*}
where $\widetilde{I}_\alpha(z)=\left(\frac{z}{2}\right)^{-\alpha}I_\alpha(z)$ and $\widetilde{K}_\alpha(z)=\left(\frac{z}{2}\right)^{-\alpha}K_\alpha(z)$ denote the renormalized classical $I$- and $K$-Bessel functions. Moreover, the corresponding measures on the complexification $\Xi=\CC^\times$ are given by $\td\nu_\lambda(z)=2|z|^{2\lambda-1}\td z$. Specializing the parameter to $\lambda=\frac{1}{2}$ and using the squaring map $\pi:\RR\setminus\{0\}\to\Omega,\,y\mapsto y^2$ we recover the classical heat equation on $\RR$ and our results translate to the corresponding results in the classical case for even functions on $\RR$ resp. $\CC$. For other parameters $\lambda\neq\frac{1}{2}$, however, our results seem to be new.\\

Finally we remark that all our results carry over to the boundary orbits of $\Omega$. For each such $G$-orbit $\calO\subseteq\partial\Omega$ there is a unique value $0\leq\lambda\leq c(\Omega)$ such that $B_\lambda$ restricts to a differential operator on $\calO$ which is self-adjoint on $L^2(\calO,\td\mu)$ for the unique (up to scalar multiples) $G$-equivariant measure $\td\mu$ on $\calO$. For this particular $\lambda$ analogous results as described above hold, involving the Bessel functions $\calI_\lambda(x,y)$ and $\calK_\lambda(x)$ on $\calO$ of parameter $\lambda$ and the $G_\CC$-orbit $G_\CC\cdot\calO\subseteq\partial\Xi$ as complexification of $\calO$ (see \cite{Moe12}).\\

\textbf{Acknowledgements.} We thank B. {\O}rsted and G. {\'O}lafsson for helpful discussions on the topic of this paper and their encouragement to publish these results.\\

\textbf{Notation.} $\NN=\{1,2,3,\ldots\}$, $\NN_0=\NN\cup\{0\}$, $\RR_+=\{x\in\RR:x>0\}$.
\section{Symmetric cones}

In this first section we collect all necessary information on symmetric cones and their natural complexifications. We further define Bessel operators and Bessel functions on symmetric cones and their complexifications. In preparation for the proofs of our main results we introduce the Segal--Bargmann transform and the unitary inversion operator which are motivated by representation theory. For details we refer the reader to \cite{FK94} and \cite{Moe12}.

\subsection{Symmetric cones and Jordan algebras}

Let $V$ be a Euclidean vectorspace of dimension $n$ with inner product $(-|-)$. An open cone $\Omega\subseteq V$ is called \textit{symmetric} if the following two conditions hold:
\begin{enumerate}
\item $\Omega$ is \textit{self-dual}, i.e. $\Omega=\Omega^*:=\{x\in V:(x|y)>0\,\forall\,y\in\overline{\Omega}\setminus\{0\}\}$,
\item $\Omega$ is \textit{homogeneous}, i.e. the automorphism group $G(\Omega):=\{g\in\GL(V):g\Omega=\Omega\}$ of $\Omega$ acts transitively on $\Omega$.
\end{enumerate}

Let $\Omega\subseteq V$ be a symmetric cone and let $G:=G(\Omega)_0$ be the identity component of its automorphism group. Then $G$ already acts transitively on $\Omega$. Put $K:=G\cap O(V)$. Then there exists a point $e\in\Omega$ such that $K=G_e=\{g\in G:ge=e\}$ and hence $\Omega\cong G/K$ is a Riemannian symmetric space. Note that $G/K$ is not semisimple since $G$ contains $\RR_+$ acting by dilations, but $\Omega\cong\RR_+\times X$, where $X$ is a semisimple Riemannian symmetric space. We decompose the Lie algebra $\frakg$ of $G$ into
\begin{align*}
 \frakg &= \frakk \oplus \frakp,
\intertext{where}
 \frakk &= \{X\in\frakg:X+X^*=0\}, & \frakp &= \{X\in\frakg:X=X^*\},
\end{align*}
$X^*$ denoting the adjoint of $X\in\End(V)$ with respect to the inner product $(-|-)$ on $V$. The mapping
\begin{align*}
 \frakp\to V,\,X\mapsto Xe
\end{align*}
is a bijection and we denote by $L:V\to\frakp$ its inverse which is symmetric with respect to the inner product $(-|-)$. If we endow $V$ with the product
\begin{align*}
 V\times V\to V,\,(x,y)\mapsto x\cdot y:=L(x)y,
\end{align*}
then $V$ becomes a \textit{Euclidean Jordan algebra} with unit element $e$. The symmetric cone $\Omega$ is the interior of the cone $\{x^2:x\in V\}$ of squares in $V$ and every $x\in\Omega$ is the square of a unique element $x^{\frac{1}{2}}\in\Omega$. The group $G$ is the identity component of the \textit{structure group} $\Str(V)$ of $V$ which is defined by
\begin{align*}
 \Str(V) &= \{g\in\GL(V):P(gx)=gP(x)g^*\,\forall\,x\in V\},
\end{align*}
where
\begin{align*}
 P(x) &:= 2L(x)^2-L(x^2), & x\in V,
\end{align*}
denotes the \textit{quadratic representation} of $V$. Further, the compact group $K$ is the identity component of the \textit{automorphism group} $\Aut(V)$ of $V$ given by
\begin{align*}
 \Aut(V) &= \{g\in\GL(V):g(x\cdot y)=gx\cdot gy\,\forall\,x,y\in V\}.
\end{align*}

We normalize the inner product $(-|-)$ on $V$ such that $(e|e)$ is equal to the rank $r=\rk(G/K)$ of the symmetric space $G/K$. The linear form $\tr(x):=(x|e)$ on $V$ is called the \textit{Jordan trace} of $V$. It is invariant under $K$ and strictly positive on $\Omega$. On the cone $\Omega$ it is equivalent to the norm function $|x|:=\sqrt{(x|x)}$:
\begin{align*}
 |x| \leq \tr(x) \leq \sqrt{r}|x|, && x\in\Omega.
\end{align*}
We further introduce the \textit{Jordan determinant} $\Delta(x)$ which is the unique homogeneous polynomial on $V$ of degree $r$ with real coefficients such that
\begin{align*}
 \Delta(x)^{\frac{2n}{r}} &= \Det(P(x)), & x\in V.
\end{align*}
Then the character $\chi(g):=\Det(g)^{\frac{r}{n}}\in\RR_+$, $g\in G$, satisfies
\begin{align*}
 \Delta(gx) &= \chi(g)\Delta(x), & g\in G,x\in V.
\end{align*}
For $\lambda\in\RR$ let
\begin{align*}
 \chi_\lambda(g) &:= \chi(g)^\lambda, & g\in G.
\end{align*}
Then every positive character of $G$ is equal to one of the $\chi_\lambda$, $\lambda\in\RR$.

\begin{example}
\begin{enumerate}
\item For $k\in\NN$ and $\FF\in\{\RR,\CC,\HH\}$ the space $V=\Herm(k,\FF)$ of Hermitian $k\times k$ matrices with entries in $\FF$ forms a Jordan algebra with multiplication
\begin{align*}
 x\cdot y &:= \frac{1}{2}(xy+yx), & x,y\in V,
\end{align*}
and unit element $e=\1$ the identity matrix. We have $G=\RR_+\PSL(k,\FF)$ and $K=\PSU(k,\FF)$ acting by
\begin{align*}
 g\cdot x &= gxg^*, & g\in G.
\end{align*}
The corresponding symmetric cone $\Omega$ is the set of all positive definite matrices in $V$. The same construction works for $\FF=\OO$ with $k\leq3$. In all cases the Jordan trace and the Jordan determinant equal the usual trace and determinant of matrices.
\item For $k\in\NN_0$ let $V=\RR^{1,k}=\RR\times\RR^k$ denote the Jordan algebra with multiplication
\begin{multline*}
 (x_1,x')\cdot(y_1,y') := (x_1y_1+\langle x',y'\rangle,x_1y'+y_1x'),\\
 (x_1,x'),(y_1,y')\in\RR\times\RR^k,
\end{multline*}
and unit element $e=(1,0,\ldots,0)$, where $\langle-,-\rangle$ is the usual inner product on $\RR^k$. The corresponding groups are given by $G=\RR_+\SO_0(1,k)$ and $K=\diag(1,\SO(k))\cong\SO(k)$. The associated symmetric cone is the Lorentz cone
\begin{align*}
 \Omega &= \{x\in\RR^{1,k}:x_1>\sqrt{x_2^2+\cdots+x_{k+1}^2}>0\}.
\end{align*}
Jordan trace and Jordan determinant are given by
\begin{align*}
 \tr(x) &= 2x_1, & \Delta(x) &= x_1^2-(x_2^2+\cdots+x_{k+1}^2).
\end{align*}
\end{enumerate}
\end{example}

These examples actually cover all simple Euclidean Jordan algebras or equivalently all irreducible symmetric cones. A list together with the corresponding groups and structure constants is given in Table \ref{tb:Groups}.

\begin{table}[H]
\begin{center}
\begin{tabular}{|c|c|c|c|c|c|}
  \cline{1-6}
  $V$ & $\frakg$ & $\frakk$ & $n$ & $r$ & $d$\\
  \hline\hline
  $\RR$ & $\RR$ & $0$ & $1$ & $1$ & $0$\\
  $\Sym(k,\RR)$ ($k\geq2$) & $\sl(k,\RR)\oplus\RR$ & $\so(k)$ & $\frac{1}{2}k(k+1)$ & $k$ & $1$\\
  $\Herm(k,\CC)$ ($k\geq2$) & $\sl(k,\CC)\oplus\RR$ & $\su(k)$ & $k^2$ & $k$ & $2$\\
  $\Herm(k,\HH)$ ($k\geq2$) & $\su^*(2k)\oplus\RR$ & $\sp(k)$ & $k(2k-1)$ & $k$ & $4$\\
  $\RR^{1,k}$ ($k\geq2$) & $\so(1,k)\oplus\RR$ & $\so(k)$ & $k+1$ & $2$ & $k-1$\\
  $\Herm(3,\OO)$ & $\mathfrak{e}_{6(-26)}\oplus\RR$ & $\mathfrak{f}_4$ & $27$ & $3$ & $8$\\
  \hline
\end{tabular}
\caption{Simple Euclidean Jordan algebras, corresponding Lie algebras and structure constants\label{tb:Groups}}
\end{center}
\end{table}

\subsection{Equivariant measures}\label{sec:EquivariantMeasures}

Let $V$ be a simple real Jordan algebra $V$ with corresponding symmetric cone $\Omega\subseteq V$. For every $\lambda\in\RR$ there is a unique (up to scalar multiples) $\chi_\lambda$-equivariant measure $\td\mu_\lambda$ on $\Omega$. Let $d\in\NN_0$ be determined by the identity
\begin{align*}
 \frac{n}{r} &= 1+(r-1)\frac{d}{2}.
\end{align*}
Then the measure $\td\mu_\lambda$ is locally finite on $\overline{\Omega}\subseteq V$ if and only if $\lambda>c(\Omega)$, where
\begin{align}
 c(\Omega) &:= \frac{n}{r}-1 = (r-1)\frac{d}{2}.\label{eq:DefcOmega}
\end{align}
We will assume this in the following. For $\lambda>c(\Omega)$ we normalize $\td\mu_\lambda$ by
\begin{align*}
 \int_\Omega{f(x)\td\mu_\lambda(x)} &= \frac{2^{r\lambda}}{\Gamma_\Omega(\lambda)}\int_\Omega{f(x)\Delta(x)^{\lambda-\frac{n}{r}}\td x}, & f\in C_c^\infty(\Omega),
\end{align*}
where
\begin{align*}
 \Gamma_\Omega(\lambda) &= (2\pi)^{\frac{n-r}{2}}\prod_{j=1}^r{\Gamma\left(\lambda-(j-1)\frac{d}{2}\right)}
\end{align*}
denotes the \textit{Gamma function} of the symmetric cone $\Omega$. This normalization is chosen such that the function $\psi_0(x)=e^{-\tr(x)}$ has norm $1$ in $L^2(\Omega,\td\mu_\lambda)$.

Denote by $\calP(\Omega)$ the space of restrictions of polynomials on $V$ to $\Omega$. Then for every $\lambda>c(\Omega)$ the space $\calP(\Omega)e^{-\tr(x)}$ is dense in $L^2(\Omega,\td\mu_\lambda)$ (see \cite[Lemma XIII.3.3]{FK94}).

\subsection{Complexifications}

The complexification $V_\CC$ of $V$ is a complex Jordan algebra. We extend the Jordan trace $\tr(-)$ to a $\CC$-linear functional on $V_\CC$. We further extend the trace form $(-|-)$ to a $\CC$-bilinear form on $V_\CC$.

The identity component of the structure group $\Str(V_\CC)$ of $V_\CC$ is in a natural way a complexification of the group $G$ and will be denoted by $G_\CC$. The orbit
\begin{align*}
 \Xi &:= G_\CC\cdot e\subseteq V_\CC
\end{align*}
is an open dense subset of $V_\CC$ and can be identified with the symmetric space $G_\CC/K_\CC$, where $K_\CC$ denotes the analytic subgroup of $G_\CC$ with Lie algebra $\frakk_\CC$. We have $\Omega\subseteq\Xi$ as a totally real submanifold. We also denote by $U$ the analytic subgroup of $G_\CC$ with Lie algebra $\frakk+i\frakp$. Then $U\subseteq G_\CC$ is the maximal compact subgroup of $G_\CC$ corresponding to the Cartan involution $g\mapsto g^{-*}=(g^*)^{-1}$, where $g^*=\overline{g}^\#$ and $g^\#$ denotes the transpose with respect to the $\CC$-bilinear trace form $(-|-)$.

Every element $z\in\Xi$ can be written as $z=ux$ with $u\in U$ and $x\in\Omega$. The formula
\begin{align}
 \int_\Xi{f(z)\td\nu_\lambda(z)} &:= \int_\Omega{\int_U{f(ux^{\frac{1}{2}})\td u}\td\mu_\lambda(x)}\label{eq:EquivMeasuresOnXi}
\end{align}
defines a one-parameter family of equivariant measures $\td\nu_\lambda$ on $\Xi$, $\lambda>c(\Omega)$. This constructs all locally finite $G_\CC$-equivariant measures on $\Xi$.

\subsection{Bessel operators on Jordan algebras}\label{sec:BesselOperators}

Let $\frac{\partial}{\partial x}$ denote the gradient with respect to the inner product $(-|-)$ of $V$. If $(e_\alpha)_\alpha$ is an orthonormal basis of $V$ and $x\in V$ is written as $x=\sum_\alpha{x_\alpha e_\alpha}$ then
\begin{align*}
 \frac{\partial f}{\partial x} &= \sum_\alpha{\frac{\partial f}{\partial x_\alpha}e_\alpha}.
\end{align*}
For $\lambda\in\CC$ we define a vector-valued second order differential operator $\calB_\lambda$ on $V$ by
\begin{align*}
 \calB_\lambda &:= P\left(\frac{\partial}{\partial x}\right)x+\lambda\frac{\partial}{\partial x}.
\end{align*}
$\calB_\lambda$ is called the \textit{Bessel operator} of parameter $\lambda$. In coordinates it be written as
\begin{align*}
 \calB_\lambda f = \sum_{\alpha,\beta}{\frac{\partial^2f}{\partial x_\alpha\partial x_\beta}P(e_\alpha,e_\beta)x}+\lambda\sum_\alpha{\frac{\partial f}{\partial x_\alpha}e_\alpha},
\end{align*}
where we used the polarized quadratic representation
\begin{align*}
 P(x,y) &= \frac{1}{2}(P(x+y)-P(x)-P(y)), & x,y\in V.
\end{align*}
Of particular importance in this article is the component of $\calB_\lambda$ of the unit element $e$:
\begin{align*}
 B_\lambda &:= (e|\calB_\lambda).
\end{align*}

Writing $z\in V_\CC$ as $z=\sum_\alpha{z_\alpha e_\alpha}$ with $z_\alpha=x_\alpha+iy_\alpha$ we extend $\calB_\lambda$ to a vector-valued holomorphic differential operator on $V_\CC$ by
\begin{align*}
 \calB_\lambda f = \sum_{\alpha,\beta}{\frac{\partial^2f}{\partial z_\alpha\partial z_\beta}P(e_\alpha,e_\beta)z}+\lambda\sum_\alpha{\frac{\partial f}{\partial z_\alpha}e_\alpha},
\end{align*}
where we use the \textit{Wirtinger derivatives}
\begin{align*}
 \frac{\partial}{\partial z_\alpha} &= \frac{1}{2}\left(\frac{\partial}{\partial x_\alpha}-i\frac{\partial}{\partial y_\alpha}\right).
\end{align*}
This also extends the scalar-valued holomorphic differential operator $B_\lambda=(e|\calB_\lambda)$.

Denote by $\ell$ the left-regular representation of $G$ or $G_\CC$ on functions on $\Omega$ or $\Xi$, i.e.
\begin{align}
 \ell(g)f(x) &= f(g^{-1}x), & g\in G\mbox{ or }G_\CC,x\in\Omega\mbox{ or }\Xi.\label{eq:DefLeftRegularRep}
\end{align}
Then the Bessel operator $\calB_\lambda$ satisfies the following equivariance property:
\begin{align}
 \ell(g)\calB_\lambda\ell(g^{-1}) &= g^\#\calB_\lambda, & g\in G\mbox{ or }G_\CC.\label{eq:BesselEquivariance}
\end{align}

\subsection{Bessel functions on Jordan algebras}\label{sec:BesselFunctions}

We briefly recall the construction of Bessel functions on Jordan algebras. For details we refer the reader to \cite[Section 3]{Moe12}.

\subsubsection*{Spherical polynomials}

Fix a Jordan frame $c_1,\ldots,c_r$ in $V$. Denote by $V_k$ the eigenspace of $L(c_1+\cdots+c_k)$ to the eigenvalue $1$. Then $V_k$ is a Euclidean Jordan subalgebra of $V$ and the orthogonal projection onto $V_k$ is given by $P(c_1+\cdots+c_k)$. Denote by $\Delta_{V_k}$ the Jordan determinant of $V_k$ and define the \textit{principal minors}
\begin{align*}
 \Delta_k(x) &:= \Delta_{V_k}(P(c_1+\cdots+c_k)x), & x\in V.
\end{align*}
For ${\bf m}\in\NN_0^r$ we say ${\bf m}\geq0$ if $m_1\geq\ldots\geq m_r\geq0$ and in this case we define the \textit{generalized power function} $\Delta_{\bf m}$ on $V$ by
\begin{align*}
 \Delta_{\bf m}(x) &:= \Delta_1(x)^{m_1-m_2}\cdots\Delta_{r-1}(x)^{m_{r-1}-m_r}\Delta_r(x)^{m_r}, & x\in V.
\end{align*}
Then $\Delta_{\bf m}$ is a polynomial on $V$ of order $|{\bf m}|=m_1+\cdots+m_r$. Let $d_{\bf m}$ denote the dimension of the vectorspace spanned by the polynomials $x\mapsto\Delta_{\bf m}(gx)$, $g\in G$.

The correponding \textit{spherical polynomials} are defined by
\begin{align*}
 \Phi_{\bf m}(x) &:= \int_K{\Delta_{\bf m}(kx)\td k}, & x\in V.
\end{align*}
By \cite[Corollary XI.3.4]{FK94} there exists a unique polynomial $\Phi_{\bf m}(z,w)$ on $V_\CC\times V_\CC$, holomorphic in $z$ and antiholomorphic in $w$, such that
\begin{align*}
 \Phi_{\bf m}(gz,w) &= \Phi_{\bf m}(z,g^*w), & z,w\in V_\CC,g\in G_\CC,\\
 \Phi_{\bf m}(x,x) &= \Phi_{\bf m}(x^2), & x\in V.
\end{align*}

\subsubsection*{$I$- and $J$-Bessel function}

We define the \textit{$I$- and $J$-Bessel function} of parameter $\lambda>c(\Omega)$ by
\begin{align*}
 \calI_\lambda(z,w) &:= \sum_{{\bf m}\geq0}{\frac{d_{\bf m}}{(\frac{n}{r})_{\bf m}(\lambda)_{\bf m}}\Phi_{\bf m}(z,w)},\\
 \calJ_\lambda(z,w) &:= \sum_{{\bf m}\geq0}{(-1)^{|{\bf m}|}\frac{d_{\bf m}}{(\frac{n}{r})_{\bf m}(\lambda)_{\bf m}}\Phi_{\bf m}(z,w)},
\end{align*}
where
\begin{align*}
 (\alpha)_{\bf m} &:= \prod_{j=1}^r{\left(\lambda-(j-1)\frac{d}{2}\right)_{m_j}}
\end{align*}
denotes the \textit{generalized Pochhammer symbol}, expressed in terms of the classical Pochhammer symbols $(\alpha)_m=\alpha(\alpha+1)\cdots(\alpha+m-1)$. For $\lambda>c(\Omega)$ we have $(\lambda)_{\bf m}\neq0$ and hence all summands are non-singular. Convergence of the series for $z,w\in V_\CC$ is proved in \cite[Lemma 3.1]{Moe12}. The functions $\calI_\lambda(z,w)$ and $\calJ_\lambda(z,w)$ are holomorphic in $z$ and antiholomorphic in $w$.

\begin{example}
For $V=\RR$ the one-dimensional Jordan algebra we have for $z,w\in\CC=V_\CC$
\begin{align*}
 \calI_\lambda(z,w) &= \widetilde{I}_{\lambda-1}(2\sqrt{z\overline{w}}), & \calJ_\lambda(z,w) &= \widetilde{J}_{\lambda-1}(2\sqrt{z\overline{w}}),
\end{align*}
where $\widetilde{I}_\alpha(z)=(\frac{z}{2})^{-\alpha}I_\alpha(z)$ and $\widetilde{J}_\alpha(z)=(\frac{z}{2})^{-\alpha}J_\alpha(z)$ denote the renormalized $I$- and $J$-Bessel functions on $\CC$. Note that $\widetilde{I}_\alpha(z)$ and $\widetilde{J}_\alpha(z)$ are even entire functions and hence $\widetilde{I}_\alpha(2\sqrt{z})$ and $\widetilde{J}_\alpha(2\sqrt{z})$ are entire functions on $\CC$.
\end{example}

The following properties of the $I$- and $J$-Bessel functions can be found in \cite[Section 3]{Moe12}:

\begin{proposition}\label{prop:BesselDiffEq}
The $I$- and $J$-Bessel functions have the following properties:
\begin{enumerate}
\item They are $G_\CC$-invariant and antisymmetric in the sense that for $z,w\in V_\CC$ and $g\in G_\CC$ we have
\begin{align*}
 \calI_\lambda(gz,w) &= \calI_\lambda(z,g^*w), & \calJ_\lambda(gz,w) &= \calJ_\lambda(z,g^*w),\\
 \calI_\lambda(z,w) &= \overline{\calI_\lambda(w,z)}, & \calJ_\lambda(z,w) &= \overline{\calJ_\lambda(w,z)}.
\end{align*}
\item For both $u(z,w)=\calI_\lambda(z,w)$ and $u(z,w)=\calJ_\lambda(z,w)$ there exists a constant $C>0$ such that following estimate holds:
\begin{align*}
 |u(z,w)| &\leq C(1+|z|\cdot|w|)^{\frac{r(2n-1)}{4}}e^{2r\sqrt{|z|\cdot|w|}}, & z,w\in V_\CC.
\end{align*}
\item They solve the following differential equations
\begin{align*}
 (\calB_\lambda)_z\calI_\lambda(z,w) &= \overline{w}\calI_\lambda(z,w), & (\calB_\lambda)_z\calJ_\lambda(z,w) &= -\overline{w}\calJ_\lambda(z,w).
\end{align*}
\end{enumerate}
\end{proposition}

\begin{remark}\label{rem:MultivariableBesselFunctions}
Note that both $\calI_\lambda(z,w)$ and $\calJ(z,w)$ are because of Proposition \ref{prop:BesselDiffEq}~(1) uniquely determined by the functions $\calI_\lambda(x):=\calI_\lambda(x,e)$ and $\calJ_\lambda(x):=\calJ_\lambda(x,e)$. The $J$-Bessel function $\calJ_\lambda(x)$ has been studied before \cite{Dib90,FT87}. Since both $\calI_\lambda(x)$ and $\calJ_\lambda(x)$ are $K$-invariant and $\Omega=K\sum_{i=1}^r{\RR_+c_i}$ these functions are uniquely determined by their values on $\sum_{i=1}^r{\RR_+c_i}$. Now $K$ contains all permutations of $\{c_1,\ldots,c_r\}$ and hence one can view $\calI_\lambda(x)$ and $\calJ_\lambda(x)$ as symmetric functions of $r$ variables $a_1,\ldots,a_r\in\RR_+$ by $(a_1,\ldots,a_r)\mapsto\sum_{i=1}^r{a_ic_i}\in\Omega$. Therefore we can consider the $I$- and $J$-Bessel function on $\Omega$ as a multivariable generalization of the classical one-variable $I$- and $J$-Bessel function.
\end{remark}

\subsubsection*{$K$-Bessel function}

We define the \textit{$K$-Bessel function} of parameter $\lambda>c(\Omega)$ by the convergent integral
\begin{align*}
 \calK_\lambda(x) &:= \int_\Omega{e^{-\tr(u^{-1})-(x|u)}\Delta(u)^{\lambda-\frac{2n}{r}}\td u}, & x\in\Omega.
\end{align*}
The function $\calK_\lambda(x)$ is strictly positive on $\Omega$ and satisfies the differential equation
\begin{align*}
 \calB_\lambda\calK_\lambda(x) &= e\calK_\lambda(x), & x\in\Omega.
\end{align*}

\begin{example}
For $V=\RR$ the function $\calK_\lambda(x)$ is given by the renormalized classical $K$-Bessel function $\widetilde{K}_\alpha(z):=\left(\frac{z}{2}\right)^{-\alpha}K_\alpha(z)$ (see \cite[formula 3.471~(9)]{GR65}):
\begin{align*}
 \calK_\lambda(x) &= 2\widetilde{K}_{\lambda-1}(2\sqrt{x}), & x>0.
\end{align*}
\end{example}

Since $\calK_\lambda(x)$ is also $K$-invariant one can by Remark \ref{rem:MultivariableBesselFunctions} interpret it as a multivariable generalization of the classical one-variable $K$-Bessel function.

Using the decomposition $\Xi=U\Omega$ we define a $U$-invariant function $\omega_\lambda$ on $\Xi$ by
\begin{align*}
 \omega_\lambda(ux) &:= \calK_\lambda\left(\left(\frac{x}{2}\right)^2\right), & z\in U,x\in\Omega.
\end{align*}
Correspondingly the function $\omega_\lambda(z)$ satisfies the differential equation
\begin{align*}
 \calB_\lambda\omega_\lambda(z) &= \frac{\overline{z}}{4}\omega_\lambda(z), & z\in\Xi.
\end{align*}

\subsection{Segal--Bargmann transform and unitary inversion operator}\label{sec:SegalBargmannAndUnitaryInversion}

The group $\Aut(T_\Omega)$ of holomorphic automorphisms of the tube domain $T_\Omega:=V+i\Omega\subseteq V_\CC$ is a Hermitian Lie group. Its scalar type holomorphic discrete series can be analytically continued to a family of highest weight representations of its universal covering group $\widetilde{\Aut(T_\Omega)}$. These are parameterized by the so-called \textit{Wallach set} whose continuous part belongs to highest weight representations realized on $L^2(\Omega,\td\mu_\lambda)$ for $\lambda>c(\Omega)$. The longest Weyl group element essentially acts in this realization as the \textit{unitary inversion operator} $\calU_\lambda$ on $L^2(\Omega,\td\mu_\lambda)$ which is the unitary involutive isomorphism given by (see \cite[Theorem 6.3]{Moe12})
\begin{align}
 \calU_\lambda f(x) &= 2^{-r\lambda}\int_\Omega{\calJ_\lambda(x,y)f(y)\td\mu(y)}.\label{eq:DefUnitaryInversion}
\end{align}

In \cite{HKMO12,Moe12} there was established another realization of the same representations on a weighted Bergmann space on $\Xi$. Since $\Xi$ is open dense in $V_\CC$ every holomorphic function on $V_\CC$ is already uniquely determined by its values on $\Xi$. We let $\calO(\overline{\Xi})$ be the space of restrictions of holomorphic functions on $V_\CC$ to $\Xi$. For $F,G\in\calO(\overline{\Xi})$ put
\begin{align*}
 \langle F,G\rangle_\lambda &:= \frac{1}{2^{3r\lambda}\Gamma_\Omega(\frac{n}{r})}\int_\Xi{F(z)\overline{G(z)}\omega_\lambda(z)\td\nu_\lambda(z)}
\end{align*}
whenever the integral converges. Let $\calF_\lambda(\Xi)$ be the space of $F\in\calO(\overline{\Xi})$ such that $\langle F,F\rangle_\lambda<\infty$ and endow it with the inner product $\langle-,-\rangle_\lambda$.

\begin{theorem}[{\cite[Theorems 4.15, 5.4 \&\ 5.7]{Moe12}}]\label{thm:FockSpace}
$\calF_\lambda(\Xi)$ is a reproducing kernel Hilbert space with the following properties:
\begin{enumerate}
\item[(a)] The reproducing kernel of $\calF_\lambda(\Xi)$ is given by $\KK_\lambda(z,w)=\calI_\lambda\left(\frac{z}{2},\frac{w}{2}\right)$,
\item[(b)] The space $\calP(V_\CC)$ of holomorphic polynomials on $V_\CC$ is dense in $\calF_\lambda(\Xi)$.
\item[(c)] The \textit{Segal--Bargmann transform} $\BB_\lambda:L^2(\Omega,\td\mu_\lambda)\to\calF_\lambda(\Xi)$ given by
\begin{align*}
 \BB_\lambda f(z) &:= e^{-\frac{1}{2}\tr(z)}\int_\Omega{\calI_\lambda(z,x)e^{-\tr(x)}f(x)\td\mu_\lambda(x)}.
\end{align*}
is a unitary isomorphism.
\end{enumerate}
\end{theorem}

There exists a unitary representation of $\widetilde{\Aut(T_\Omega)}$ on the Hilbert space $\calF_\lambda(\Xi)$ which is isomorphic to the unitary highest weight representation on $L^2(\Omega,\td\mu_\lambda)$. In fact, the two models $L^2(\Omega,\td\mu_\lambda)$ and $\calF_\lambda(\Xi)$ are related by the Segal--Bargmann transform in the sense that it intertwines the group actions. We only use the intertwining property for the longest Weyl group element which acts on $\calF_\lambda(\Xi)$ essentially as
\begin{align*}
 (-1)^*F(z) &:= F(-z), & F\in\calF_\lambda(\Xi).
\end{align*}
This corresponds to the following intertwining relation (see \cite[Proposition 6.6]{Moe12}):
\begin{align}
 \BB_\lambda\circ\calU_\lambda &= (-1)^*\circ\BB_\lambda.\label{eq:IntRelSBTrafoUnitInv}
\end{align}
Using this intertwining relation we derive two integral formulas for the $I$-Bessel function:

\begin{lemma}\label{lem:RepPropIBessel}
\begin{enumerate}
\item For $z\in V_\CC$ we have
\begin{align*}
 2^{-r\lambda}\int_\Omega{e^{-\tr(\xi)}\calI_\lambda(z,\xi)\td\mu_\lambda(\xi)} &= e^{\tr(z)}.
\end{align*}
\item For $x\in\Omega$ and $z\in V_\CC$ the we have the reproducing identity
\begin{align*}
 2^{-r\lambda}\int_\Omega{e^{-\tr(\xi)}\calI_\lambda(z,\xi)\calI_\lambda(-x,\xi)\td\mu_\lambda(\xi)} &= e^{\tr(z)-\tr(x)}\calI_\lambda(-z,x).
\end{align*}
\end{enumerate}
\end{lemma}

\begin{proof}
\begin{enumerate}
\item Since $\calI_\lambda(0,z)=1$ this immediately follows from (2) by putting $x=0$.
\item For any $\varphi\in C_c^\infty(\Omega)$ and $z\in V_\CC$ we have
\begin{align*}
 \BB_\lambda\calU_\lambda\varphi(z) &= e^{-\frac{1}{2}\tr(z)}\int_\Omega{\calI_\lambda(z,x)e^{-\tr(x)}\calU_\lambda\varphi(x)\td\mu_\lambda(x)}\\
 &= 2^{-r\lambda}e^{-\frac{1}{2}\tr(z)}\int_\Omega{\int_\Omega{\calI_\lambda(z,x)\calJ_\lambda(x,y)e^{-\tr(x)}\varphi(y)\td\mu_\lambda(y)}\td\mu_\lambda(x)}\\
 &= 2^{-r\lambda}e^{-\frac{1}{2}\tr(z)}\int_\Omega{\int_\Omega{e^{-\tr(x)}\calI_\lambda(z,x)\calI_\lambda(-x,y)\td\mu_\lambda(x)}\varphi(y)\td\mu_\lambda(y)}.
\end{align*}
On the other hand, using the intertwining relation \eqref{eq:IntRelSBTrafoUnitInv} we obtain
\begin{align*}
 \BB_\lambda\calU_\calO\varphi(z) &= \BB_\lambda\varphi(-z)\\
 &= e^{\frac{1}{2}\tr(z)}\int_\Omega{\calI_\lambda(-z,y)e^{-\tr(y)}\varphi(y)\td\mu_\lambda(y)}.
\end{align*}
Therefore, the integral kernels have to coincide, which gives
\begin{align*}
 2^{-r\lambda}e^{-\frac{1}{2}\tr(z)}\int_\Omega{e^{-\tr(x)}\calI_\lambda(z,x)\calI_\lambda(-x,y)\td\mu_\lambda(x)} &= e^{\frac{1}{2}\tr(z)}\calI_\lambda(-z,y)e^{-\tr(y)}.
\end{align*}
This is the claimed formula.\qedhere
\end{enumerate}
\end{proof}
\section{The heat kernel transform}

In this section we investigate the heat equation on $\Omega$ corresponding to the second order differential operator $B_\lambda$, find the heat kernel and characterize the image of the heat kernel transform. We also recover the heat kernel transform via the restriction principle.

\subsection{An elliptic differential operator}

Consider the second order differential operator $B_\lambda$ on $\Omega$ defined in Section \ref{sec:BesselOperators}.

\begin{proposition}
\begin{enumerate}
\item $B_\lambda$ is a $K$-invariant elliptic second order operator on $\Omega$.
\item For $\lambda>c(\Omega)$ the operator $B_\lambda$ extends to a self-adjoint operator on $L^2(\Omega,\td\mu_\lambda)$.
\end{enumerate}
\end{proposition}

\begin{proof}
\begin{enumerate}
\item $K$-invariance of $B_\lambda$ is clear by \eqref{eq:BesselEquivariance} since $K$ fixes the identity element $e$ of $V$. To show ellipticity we identify for $x\in\Omega$ the cotangent space $T_x^*\Omega$ with $V$. Then the principal symbol of $B_\lambda$ at $x\in\Omega$ in direction $\xi\in T_x^*\Omega$ is given by
\begin{align*}
 (P(\xi)x|e) &= (x|P(\xi)e) = (x|\xi^2) = (L(x)\xi|\xi).
\end{align*}
By \cite[Proposition III.2.2]{FK94} the operator $L(x)$ is positive definite since $x\in\Omega$ and hence the claim follows.
\item By \cite[Theorem 3.4]{ADO06} the operator $iB_\lambda$ appears in the Lie algebra action of a unitary representation on the Hilbert space $L^2(\Omega,\td\mu_\lambda)$. Therefore, $iB_\lambda$ extends to a skew-adjoint operator on $L^2(\Omega,\td\mu_\lambda)$ which shows the claim.\qedhere
\end{enumerate}
\end{proof}

The spectral decomposition of the operator $B_\lambda$ can also be written down explicitly in terms of the $J$-Bessel functions $\calJ_\lambda(x,y)$ defined in Section \ref{sec:BesselFunctions}. Recall from Proposition \ref{prop:BesselDiffEq} that $\calJ_\lambda(-,y)$ is an eigenfunction of $B_\lambda$ to the eigenvalue $-\tr(y)$.

\begin{proposition}\label{prop:SpectralDecompBessel}
\begin{enumerate}
\item The spectrum of $B_\lambda$ on $L^2(\Omega,\td\mu_\lambda)$ is purely continuous and given by $(-\infty,0]$.
\item The spectral decomposition of $B_\lambda$ is given by
\begin{align*}
 f(x) &= 2^{-r\lambda}\int_\Omega{\calJ_\lambda(x,y)\calU_\lambda f(y)\td\mu_\lambda(y)},
\end{align*}
where $\calU_\lambda$ is the unitary isomorphism on $L^2(\Omega,\td\mu_\lambda)$ defined in \eqref{eq:DefUnitaryInversion}. In particular $\calU_\lambda$ intertwines the differential operator $B_\lambda$ with the multiplication operator $-\tr(x)$.
\end{enumerate}
\end{proposition}

\begin{proof}
Clearly (1) follows from (2). The statements of (2) can be found in \cite[Section 6]{Moe12}.
\end{proof}

\begin{remark}
The Laplace operator $\Delta$ of the Riemannian symmetric space $\Omega\cong G/K$ is related to the operator $B_\lambda$ in the following way
\begin{align*}
 \Delta &= (x|\calB_\lambda)-\left(\lambda-\frac{n}{r}\right)\calE,\\
 B_\lambda &= (e|\calB_\lambda),
\end{align*}
where $\calE=\left(x\middle|\frac{\partial}{\partial x}\right)$ is the Euler operator. Therefore, both $\Delta$ and $B_\lambda$ are polynomial second-order differential operators on $\Omega$ with the important difference that $B_\lambda$ is of Euler degree $-1$ whereas $\Delta$ is of Euler degree $0$.
\end{remark}

\subsection{The heat kernel}

We consider the following initial value problem for the heat equation on $(0,\infty)\times\Omega$ corresponding to the operator $B_\lambda$:
\begin{align}
 ((B_\lambda)_x-\partial_t)u(t,x) &= 0 && \mbox{on $(0,\infty)\times\Omega$,}\label{eq:HeatEquation}\\
 u(0,x) &= f(x) && \mbox{on $\Omega$}\label{eq:InitialValue}
\end{align}
for some function $f$ on $\Omega$. This initial value problem can be solved using the heat kernel. For this recall the $I$-Bessel function $\calI_\lambda(x,y)$ defined in Section \ref{sec:BesselFunctions}. We define the \textit{heat kernel} by
\begin{align*}
 h_\lambda(t,x,y) &= (2t)^{-r\lambda}e^{-\frac{1}{t}(\tr(x)+\tr(y))}\calI_\lambda\left(\frac{x}{t},\frac{y}{t}\right), & t>0,x,y\in\Omega.
\end{align*}
Note that $h_\lambda(t,x,y)>0$ for $t>0$ and $x,y\in\Omega$.

\begin{proposition}\label{prop:HeatKernelProperties1}
\begin{enumerate}
\item For each $\lambda>c(\Omega)$ there exists a constant $C>0$ such that
\begin{align*}
 h_\lambda(t,x,y) &\leq Ct^{-r\lambda}\left(1+\frac{|x|\cdot|y|}{t^2}\right)^{\frac{r(2n-1)}{4}}e^{-\frac{1}{t}(\tr(x)+\tr(y)-2r\sqrt{|x|\cdot|y|})}
\end{align*}
for all $t>0$, $x,y\in\Omega$.
\item The heat kernel satisfies the following invariance property for $t>0$, $x,y\in\Omega$ and $k\in K$:
\begin{align*}
 h_\lambda(t,kx,y) &= h_\lambda(t,x,k^{-1}y).
\end{align*}
\item The heat kernel is symmetric in $x$ and $y$:
\begin{align*}
 h_\lambda(t,x,y) &= h_\lambda(t,y,x).
\end{align*}
\item For $t>0$ and $x,y\in\Omega$ we have
\begin{align}
 h_\lambda(t,x,y) = 2^{-2r\lambda}\int_\Omega{e^{-t\cdot\tr(\xi)}\calI_\lambda(-x,\xi)\calI_\lambda(-y,\xi)\td\mu_\lambda(\xi)}.\label{eq:HeatKernelRepFormula}
\end{align}
\end{enumerate}
\end{proposition}

\begin{proof}
\begin{enumerate}
\item This follows immediately from Proposition \ref{prop:BesselDiffEq}~(2).
\item Since for $k\in K$ we have $k^*=k^{-1}$ this is Proposition \ref{prop:BesselDiffEq}~(3).
\item Since the $I$-Bessel function is symmetric by Proposition \ref{prop:BesselDiffEq}~(1) this is clear.
\item This is immediate with Lemma \ref{lem:RepPropIBessel}.\qedhere
\end{enumerate}
\end{proof}

\begin{proposition}\label{prop:HeatKernelProperties2}
The heat kernel $h_\lambda(t,x,y)$ has the following properties for $s,t>0$ and $x,y\in\Omega$
\begin{enumerate}
\item (Normalization)
\begin{align*}
 \int_\Omega{h_\lambda(t,x,y)\td\mu_\lambda(y)} &= 1,
\end{align*}
\item (Semigroup)
\begin{align*}
 \int_\Omega{h_\lambda(s,x,z)h_\lambda(t,y,z)\td\mu_\lambda(z)} &= h_\lambda(s+t,x,y),
\end{align*}
\item (Differential equation) For every $y\in\Omega$ the function $h_\lambda(t,x,y)$ solves the heat equation \eqref{eq:HeatEquation} on $(0,\infty)\times\Omega$.
\end{enumerate}
\end{proposition}

\begin{proof}
\begin{enumerate}
\item We have
\begin{align*}
 \int_\Omega{h_\lambda(t,x,y)\td\mu_\lambda(y)} &= (2t)^{-r\lambda}e^{-\frac{1}{t}\tr(x)}\int_\Omega{\calI_\lambda\left(\frac{x}{t},\frac{y}{t}\right)e^{-\frac{1}{t}\tr(y)}\td\mu_\lambda(y)}\\
\intertext{and substituting $z=\frac{y}{t}$ we obtain}
 &= 2^{-r\lambda}e^{-\frac{1}{t}\tr(x)}\int_\Omega{\calI_\lambda\left(\frac{x}{t},z\right)e^{-\tr(z)}\td\mu_\lambda(z)} = 1
\end{align*}
by Lemma \ref{lem:RepPropIBessel}~(1).
\item We substitute \eqref{eq:HeatKernelRepFormula} for the first factor in the integrand. This yields
\begin{align*}
 & \int_\Omega{h_\lambda(s,x,z)h_\lambda(t,y,z)\td\mu_\lambda(z)}\\
 ={}& (8t)^{-r\lambda}\int_\Omega{\int_\Omega{e^{-s\tr(\xi)}\calI_\lambda(-x,\xi)\calI_\lambda(-z,\xi)e^{-\frac{1}{t}(\tr(y)+\tr(z))}\calI_\lambda\left(\frac{y}{t},\frac{z}{t}\right)\td\mu_\lambda(\xi)}\td\mu_\lambda(z)}\\
\intertext{and substituting $z=t\eta$ gives}
 ={}& 8^{-r\lambda}\int_\Omega{\left(\int_\Omega{e^{-\tr(\eta)}\calI_\lambda(-\eta,t\xi)\calI_\lambda\left(\frac{y}{t},\eta\right)\td\mu_\lambda(\eta)}\right)e^{-\frac{1}{t}\tr(y)}e^{-s\tr(\xi)}\calI_\lambda(-x,\xi)\td\mu_\lambda(\xi)}.\\
\intertext{Now Lemma \ref{lem:RepPropIBessel}~(2) gives}
 ={}& 2^{-2r\lambda}\int_\Omega{e^{-(s+t)\tr(\xi)}\calI_\lambda(-x,\xi)\calI_\lambda(-y,\xi)\td\mu_\lambda(\xi)}\\
 ={}& h_\lambda(s+t,x,y)
\end{align*}
by \eqref{eq:HeatKernelRepFormula} again.
\item This follows from \eqref{eq:HeatKernelRepFormula} by differentiating under the integral and using Proposition \ref{prop:BesselDiffEq}~(3).\qedhere
\end{enumerate}
\end{proof}

We also need some observations on the holomorphic extension of the heat kernel.

\begin{proposition}\label{prop:HeatKernelAnalyticExtension}
The heat kernel $h_\lambda(t,x,y)$ admits a unique extension to a function $h_\lambda(t,z,w)$ on $\RR_+\times V_\CC\times V_\CC$ which is holomorphic in $z$ and antiholomorphic in $w$. It further has the following properties for $t>0$ and $z,w\in V_\CC$:
\begin{enumerate}
\item For each $\lambda>c(\Omega)$ there exists a constant $C>0$ such that
\begin{align*}
 h_\lambda(t,z,w) &\leq Ct^{-r\lambda}\left(1+\frac{|z|\cdot|w|}{t^2}\right)^{\frac{r(2n-1)}{4}}e^{-\frac{1}{t}(\Re\tr(z)+\Re\tr(w)-2r\sqrt{|z|\cdot|w|})}
\end{align*}
for all $t>0$, $x,y\in\Omega$,
\item $\overline{h_\lambda(t,z,w)}=h_\lambda(t,w,z)$,
\item $h_\lambda(t,kz,w)=h_\lambda(t,z,k^{-1}w)$, $k\in K$.
\end{enumerate}
\end{proposition}

\begin{proof}
Define
\begin{align*}
 h_\lambda(t,z,w) &= (2t)^{-r\lambda}e^{-\frac{1}{t}(\tr(z)+\tr(\overline{w}))}\calI_\lambda\left(\frac{z}{t},\frac{w}{t}\right), & t>0,z,w\in V_\CC,
\end{align*}
then everything follows from the results in Section \ref{sec:BesselFunctions}.
\end{proof}

\subsection{The heat semigroup}

For $t>0$ we define the \textit{heat kernel transform} $H_\lambda(t)$ on $f\in C_c^\infty(\Omega)$ by
\begin{align}
 H_\lambda(t)f(x) &:= \int_\Omega{h_\lambda(t,x,y)f(y)\td\mu_\lambda(y)}, & x\in\Omega.\label{eq:DefHeatKernelTrafo}
\end{align}

\begin{theorem}\label{thm:HeatKernelTransform}
For $t>0$ and $1\leq p\leq\infty$ the integral in \eqref{eq:DefHeatKernelTrafo} converges absolutely for every $f\in L^p(\Omega,\td\mu_\lambda)$ and defines a continuous linear operator on $L^p(\Omega,\td\mu_\lambda)$ of operator norm $\leq1$ with the following properties:
\begin{enumerate}
\item $H_\lambda(t)$ is smoothing, i.e. $H_\lambda(t)f\in C^\infty(\Omega)$ for $f\in L^p(\Omega,\td\mu_\lambda)$,
\item $H_\lambda(s)\circ H_\lambda(t)=H_\lambda(s+t)$, $s,t>0$,
\item $H_\lambda(t)$ is $K$-equivariant, i.e. $H_\lambda(t)\circ\ell(k)=\ell(k)\circ H_\lambda(t)$ with $\ell(k)$ as in \eqref{eq:DefLeftRegularRep},
\item $H_\lambda(t)$ is symmetric on $L^2(\Omega,\td\mu_\lambda)$.
\end{enumerate}
\end{theorem}

\begin{proof}
First note that for fixed $t>0$ and $x\in\Omega$ the function $h_\lambda(t,x,y)$ is by Proposition \ref{prop:HeatKernelProperties1}~(1) contained in $L^q(\Omega,\td\mu_\lambda)$ for any $1\leq q\leq\infty$ with $L^q$-norm depending continuously on $x$. Hence the integral in \eqref{eq:DefHeatKernelTrafo} converges absolutely for every $f\in L^p(\Omega,\td\mu_\lambda)$ and defines a smooth function $H_\lambda(t)f$ on $\Omega$. Assume that $1<p<\infty$. To show that $H_\lambda(t)f\in L^p(\Omega,\td\mu_\lambda)$ first observe that by H\"older's inequality with $1<q<\infty$ the dual exponent to $p$, i.e. $\frac{1}{p}+\frac{1}{q}=1$, we find
\begin{align*}
 & \int_\Omega{|h_\lambda(t,x,y)f(y)|\td\mu_\lambda(y)}\\
 ={}& \int_\Omega{h_\lambda(t,x,y)^{\frac{1}{q}}\cdot h_\lambda(t,x,y)^{\frac{1}{p}}|f(y)|\td\mu_\lambda(y)}\\
 \leq{}& \left(\int_\Omega{h_\lambda(t,x,y)\td\mu_\lambda(y)}\right)^{\frac{1}{q}}\left(\int_\Omega{h_\lambda(t,x,y)|f(y)|^p\td\mu_\lambda(y)}\right)^{\frac{1}{p}}\\
 ={}& \left(\int_\Omega{h_\lambda(t,x,y)|f(y)|^p\td\mu_\lambda(y)}\right)^{\frac{1}{p}}.
\end{align*}
where we have used Proposition \ref{prop:HeatKernelProperties2}\,(1). Then we find, using Fubini's theorem
\begin{align*}
 \|H_\lambda(t)f\|_{L^p(\Omega,\td\mu_\lambda)}^p &= \int_\Omega{\left|\int_\Omega{h_\lambda(t,x,y)f(y)\td\mu_\lambda(y)}\right|^p\td\mu_\lambda(x)}\\
 &\leq \int_\Omega{\int_\Omega{h_\lambda(t,x,y)|f(y)|^p\td\mu_\lambda(y)}\td\mu_\lambda(x)}\\
 &= \int_\Omega{|f(y)|^p\td\mu_\lambda(y)} = \|f\|_{L^p(\Omega,\td\mu_\lambda)}^p,
\end{align*}
where we have again used Proposition \ref{prop:HeatKernelProperties2}\,(1). This shows that $H_\lambda(t)$ extends to a continuous linear operator on $L^p(\Omega,\td\mu_\lambda)$ of norm $\leq1$ for $1<p<\infty$. For $p=1$ and $p=\infty$ the proof is similar. We now prove the other properties:
\begin{enumerate}
\item This was already established above.
\item This is clear in view of Proposition \ref{prop:HeatKernelProperties2}~(2).
\item This follows from Proposition \ref{prop:HeatKernelProperties1}~(2) and the $K$-invariance of the measure $\td\mu_\lambda$.
\item This is immediate with Proposition \ref{prop:HeatKernelProperties1}~(3).\qedhere
\end{enumerate}
\end{proof}

\begin{proposition}
Let $f\in L^2(\Omega,\td\mu_\lambda)$ be of the form $f(x)=p(x)e^{-\tr(x)}$ for some polynomial $p\in\calP(\Omega)$. Define a function $u(x,t)$ on $\Omega\times[0,\infty)$ by
\begin{align*}
 u(x,t) &:= \begin{cases}H_\lambda(t)f(x) & \mbox{for $t>0$,}\\f(x) & \mbox{for $t=0$.}\end{cases}
\end{align*}
Then $u\in C^\infty(\Omega\times(0,\infty))\cap C(\Omega\times[0,\infty))$ and it solves the heat equation \eqref{eq:HeatEquation} with initial value \eqref{eq:InitialValue}.
\end{proposition}

\begin{proof}
Recall from Proposition \ref{prop:SpectralDecompBessel} the unitary involutive operator $\calU_\lambda$ on $L^2(\Omega,\td\mu_\lambda)$. From \cite[Proposition 6.2 \&\ Theorem 6.3]{Moe12} it follows that $\calU_\lambda f$ is again of the form $q(x)e^{-\tr(x)}$ for a polynomial $q\in\calP(\Omega)$. Using Proposition \ref{prop:HeatKernelProperties1}~(4) we find
\begin{align*}
 u(x,t) &= \int_\Omega{h_\lambda(t,x,y)f(y)\td\mu_\lambda(y)}\\
 &= 2^{-2r\lambda}\int_\Omega{\int_\Omega{e^{-t\tr(z)}\calJ_\lambda(x,z)\calJ_\lambda(y,z)f(y)\td\mu_\lambda(z)}\td\mu_\lambda(y)}\\
 &= 2^{-r\lambda}\int_\Omega{e^{-t\tr(z)}\calJ_\lambda(x,z)\calU_\lambda f(z)\td\mu_\lambda(z)}.
\end{align*}
Since $\calU_\lambda f(x)=q(x)e^{-\tr(x)}$ it follows from the estimate in Proposition \ref{prop:HeatKernelProperties1}~(1) that this expression defines a function in $C^\infty(\Omega\times(0,\infty))\cap C(\Omega\times[0,\infty))$. For $t=0$ it gives $\calU_\lambda^2f(x)=f(x)$ since $\calU_\lambda$ is involutive. Differentiating under the integral with Proposition \ref{prop:BesselDiffEq}~(3) finally shows the differential equation.
\end{proof}

\begin{theorem}\label{thm:HeatKernelSemigroup}
For $t>0$ we have $e^{tB_\lambda}=H_\lambda(t)$.
\end{theorem}

\begin{proof}
By Proposition \ref{prop:SpectralDecompBessel} the spectral decomposition of $B_\lambda$ is for $f\in\calP(\Omega)e^{-\tr(x)}$ given by the convergent integral
\begin{align*}
 B_\lambda f(x) &= 2^{-2r\lambda}\int_\Omega{\int_\Omega{(B_\lambda)_x\calJ_\lambda(x,y)\calJ_\lambda(y,z)f(z)\td\mu_\lambda(z)}\td\mu_\lambda(y)}\\
 &= -2^{-2r\lambda}\int_\Omega{\int_\Omega{\tr(y)\calJ_\lambda(x,y)\calJ_\lambda(y,z)f(z)\td\mu_\lambda(z)}\td\mu_\lambda(y)}.
\end{align*}
Hence $e^{tB_\lambda}$ is the integral operator
\begin{align*}
 e^{tB_\lambda}f(x) &= 2^{-2r\lambda}\int_\Omega{\int_\Omega{e^{-t\tr(y)}\calJ_\lambda(x,y)\calJ_\lambda(y,z)f(z)\td\mu_\lambda(z)}\td\mu_\lambda(y)}\\
 &= \int_\Omega{h_\lambda(t,x,z)f(z)\td\mu_\lambda(z)} = H_\lambda(t)f(x),
\end{align*}
where we have used Proposition \ref{prop:HeatKernelProperties1}~(4). Since $\calP(\Omega)e^{-\tr(x)}$ is dense in $L^2(\Omega,\td\mu_\lambda)$ the claim follows.
\end{proof}

\subsection{The image of the heat kernel transform}

To characterize the image of the heat kernel transform $H_\lambda(t)$ we first prove that every function in the image extends to a holomorphic function on $\Xi$. Recall that $\calO(\overline{\Xi})$ denotes the space of holomorphic functions on $\Xi$ which extend to $\overline{\Xi}=V_\CC$ and endow it with the topology of compact convergence.

\begin{proposition}\label{prop:HeatKernelTrafoAnalyticExtension}
Let $t>0$. Then for every $f\in L^2(\Omega,\td\mu_\lambda)$ the integral
\begin{align*}
 \widetilde{H}_\lambda(t)f(z) &:= \int_\Omega{h_\lambda(t,z,x)f(x)\td\mu_\lambda(x)}, & z\in V_\CC,
\end{align*}
converges uniformly on bounded subsets of $V_\CC$ and defines a function $\widetilde{H}_\lambda(t)f\in\calO(\overline{\Xi})$. The map $\widetilde{H}_\lambda(t):L^2(\Omega,\td\mu_\lambda)\to\calO(\overline{\Xi})$ defines a continuous linear operator.
\end{proposition}

\begin{proof}
By Proposition \ref{prop:HeatKernelAnalyticExtension}~(1) the heat kernel $h_\lambda(t,z,x)$ is for fixed $t>0$ and $z\in V_\CC$ contained in $L^2(\Omega,\td\mu_\lambda)$ with $L^2$-norm depending continuously on $z$. Since $h_\lambda(t,z,x)$ is analytic in $z$ the claim follows.
\end{proof}

We now determine the image of the heat kernel transform. Let
\begin{align*}
 \calF_{\lambda,t}(\Xi) &:= \left\{F\in\calO(\overline{\Xi}):\int_\Xi{|F(z)|^2e^{\frac{1}{t}\tr(x)}\omega_\lambda\left(\frac{z}{t}\right)\td\nu_\lambda(z)}<\infty\right\},
\end{align*}
where $z=x+iy$ with $x,y\in V$, and endow it with the inner product
\begin{align}
 \langle F,G\rangle_{\lambda,t} &:= \frac{(2t)^{-r\lambda}}{\Gamma_\Omega(\frac{n}{r})}\int_\Xi{F(z)\overline{G(z)}e^{\frac{1}{t}\tr(x)}\omega_\lambda\left(\frac{z}{t}\right)\td\nu_\lambda(z)}.\label{eq:BergmannSpaceInnerProduct}
\end{align}

\begin{theorem}\label{thm:HeatKernelTrafoImage}
\begin{enumerate}
\item The space $\calF_{\lambda,t}(\Xi)$ is a Hilbert space with reproducing kernel
\begin{align*}
 \KK_{\lambda,t}(z,w) &= h_\lambda(2t,z,w) = \frac{1}{(4t)^{r\lambda}}e^{-\frac{1}{2t}(\tr(z)+\tr(\overline{w}))}\calI_\lambda\left(\frac{z}{2t},\frac{w}{2t}\right).
\end{align*}
\item The space
\begin{align*}
 \{p(z)e^{-\frac{1}{2t}\tr(z)}:\mbox{$p$ a holomorphic polynomial on $V_\CC$}\}
\end{align*}
is contained in $\calF_{\lambda,t}(\Xi)$ and forms a dense subspace.
\item The heat kernel transform is a unitary isomorphism
\begin{align*}
 \widetilde{H}_\lambda(t):L^2(\Omega,\td\mu_\lambda)\to\calF_{\lambda,t}(\Xi).
\end{align*}
\end{enumerate}
\end{theorem}

\begin{proof}
Recall the Hilbert space $\calF_\lambda(\Xi)$ and the Segal--Bargmann transform $\BB_\lambda$ from Section \ref{sec:SegalBargmannAndUnitaryInversion}. We define a unitary isomorphism $\Phi_t:\calF_\lambda(\Xi)\to\calF_{\lambda,t}(\Xi)$ by
\begin{align*}
 \Phi_tF(z) &:= (4t)^{-\frac{1}{2}r\lambda}e^{-\frac{1}{2t}\tr(z)}F(\tfrac{z}{t}).
\end{align*}
Then statements (1) and (2) follow immediately from Theorem \ref{thm:FockSpace}~(a) and (b). To see that statement (3) follows from Theorem \ref{thm:FockSpace}~(c) we note that
\begin{align*}
 \tau_tf(x) &:= t^{\frac{r\lambda}{2}}f(tx)
\end{align*}
defines a unitary isomorphism on $L^2(\Omega,\td\mu_\lambda)$ and then it is easy to see that
\begin{align*}
 \widetilde{H}_\lambda(t) &= \Phi_t\circ\BB_\lambda\circ\tau_t.
\end{align*}
This shows (3) and the proof is complete.
\end{proof}

\subsection{The restriction principle}\label{sec:RestrictionPrinciple}

The heat kernel transform $\widetilde{H}_\lambda(t)$ can also be obtained via the restriction principle. This gives an alternative proof of its unitarity.

\begin{lemma}
The restriction operator
\begin{align*}
 \calR_{\lambda,t}:\calD(\calR_{\lambda,t})\to L^2(\Omega,\td\mu_\lambda),\,F\mapsto F|_\Omega,
\end{align*}
with domain
\begin{align*}
 \calD(\calR_{\lambda,t}) &= \{F\in\calF_{\lambda,t}(\Xi):F|_\Omega\in L^2(\Omega,\td\mu_\lambda)\}
\end{align*}
is a closed unbounded operator. It is densely defined, injective and has dense range.
\end{lemma}

\begin{proof}
To show that $\calR_{\lambda,t}$ is closed let $(F_n)_n\subseteq\calD(\calR_{\lambda,t})$ be a sequence with $F_n\to F$ in $\calF_{\lambda,t}(\Xi)$ as $n\to\infty$ and $\calR_{\lambda,t}F_n\to f$ in $L^2(\Omega,\td\mu_\lambda)$. Since point evaluations in $\calF_{\lambda,t}(\Xi)$ are continuous it follows that in particular $F_n(x)\to F(x)$ for every $x\in\Omega$. Hence $f=F|_\Omega\in L^2(\Omega,\td\mu_\lambda)$.\\
The domain $\calD(\calR_{\lambda,t})$ certainly contains the space
\begin{align*}
 \{p(z)e^{-\frac{1}{2t}\tr(z)}:\mbox{$p$ a holomorphic polynomial on $V_\CC$}\}
\end{align*}
which is by Theorem \ref{thm:HeatKernelTrafoImage}~(2) dense in $\calF_{\lambda,t}(\Xi)$. Its image under $\calR_{\lambda,t}$ is equal to $\calP(\Omega)e^{-\frac{1}{2t}\tr(x)}$ which is dense in $L^2(\Omega,\td\mu_\lambda)$. Hence $\calR_{\lambda,t}$ is densely defined and has dense range. Finally injectivity of $\calR_{\lambda,t}$ is clear since $\calF_{\lambda,t}(\Xi)$ consists of holomorphic functions and $\Omega\subseteq\Xi$ is totally real.
\end{proof}

We consider the adjoint $\calR_{\lambda,t}^*:\calD(\calR_{\lambda,t}^*)\to\calF_{\lambda,t}(\Xi)$.

\begin{proposition}
The operator $\calR_{\lambda,t}^*$ is given by
\begin{align*}
 \calR_{\lambda,t}^*f(z) &= \int_\Omega{h_\lambda(2t,z,x)f(x)\td\mu_\lambda(x)} = \widetilde{H}_\lambda(2t)f(z).
\end{align*}
In particular, $\calR_{\lambda,t}\calR_{\lambda,t}^*$ is bounded on $L^2(\Omega,\td\mu_\lambda)$ and agrees with the heat kernel transform $H_\lambda(2t)$.
\end{proposition}

Note that $\widetilde{H}_\lambda(2t):L^2(\Omega,\td\mu_\lambda)\to\calF_{\lambda,2t}(\Xi)$ is an isomorphism and we have
\begin{align*}
 \calD(\calR_{\lambda,t}^*) &= \{f\in L^2(\Omega,\td\mu_\lambda):\widetilde{H}_\lambda(2t)f\in\calF_{\lambda,t}(\Xi)\}.
\end{align*}

\begin{proof}
We have
\begin{align*}
 \calR_{\lambda,t}^*f(x) &= \langle\calR_{\lambda,t}^*f|\KK_{\lambda,t}(-,z)\rangle = \langle f|\calR_{\lambda,t}\KK_{\lambda,t}(-,z)\rangle\\
 &= \int_\Omega{f(x)\overline{\KK_{\lambda,t}(x,z)}\td\mu_\lambda(x)} = \int_\Omega{h_\lambda(2t,z,x)f(x)\td\mu_\lambda(x)}.\qedhere
\end{align*}
\end{proof}

Now consider the polar decomposition of the operator $\calR_{\lambda,t}^*$:
\begin{align*}
 \calR_{\lambda,t}^* &= \calU_{\lambda,t}\circ|\calR_{\lambda,t}|
\end{align*}
with a unitary operator $\calU_{\lambda,t}:L^2(\Omega,\td\mu_\lambda)\to\calF_{\lambda,t}(\Xi)$ and $|\calR_{\lambda,t}|=\sqrt{\calR_{\lambda,t}\calR_{\lambda,t}^*}$. Since $\calR_{\lambda,t}\calR_{\lambda,t}^*=H_\lambda(2t)$ and $(H_\lambda(t))_{t>0}$ forms a semigroup the square root is given by $|\calR_{\lambda,t}|=H_\lambda(t)$.

\begin{theorem}
$$ \calU_{\lambda,t} = \widetilde{H}_\lambda(t). $$
\end{theorem}

\begin{proof}
We have
\begin{align*}
 \left(\calU_{\lambda,t}\circ|\calR_{\lambda,t}|\right)f(z) &= \calR_{\lambda,t}^*f(z) = \widetilde{H}_\lambda(2t)f(z)\\
 &= \left(\widetilde{H}_\lambda(t)\circ H_\lambda(t)\right)f(z) = \left(\widetilde{H}_\lambda(t)\circ|\calR_{\lambda,t}|\right)f(z)
\end{align*}
and hence $\calU_{\lambda,t}=\widetilde{H}_\lambda(t)$.
\end{proof}
\section{Example: The positive real line}\label{sec:Example}

We illustrate our results at the example $\Omega=\RR_+$. Here $V=\RR$ is the one-dimensional Jordan algebra with complexification $V_\CC=\CC$ and the complexification $\Xi$ of $\Omega$ is given by $\Xi=\CC^\times=\CC\setminus\{0\}$. The Bessel operator $B_\lambda$ is given by
\begin{align*}
 B_\lambda &= x\frac{\td^2}{\td x^2}+\lambda\frac{\td}{\td x},
\end{align*}
and $L^2(\Omega,\td\mu_\lambda)=L^2(\RR_+,x^{\lambda-1}\td x)$, $\lambda>c(\Omega)=0$. The heat kernel takes the form
\begin{align*}
 h_\lambda(t,z,w) &= (2t)^{-\lambda}e^{-\frac{z+\overline{w}}{t}}\widetilde{I}_{\lambda-1}\left(2\frac{\sqrt{z\overline{w}}}{t}\right), & t>0,z,w\in\CC.
\end{align*}
Hence the heat kernel transform is given by
\begin{align*}
 \widetilde{H}_\lambda(t)f(z) &= (2t)^{-\lambda}e^{-\frac{z}{t}}\int_0^\infty{\widetilde{I}_{\lambda-1}\left(2\frac{\sqrt{zx}}{t}\right)e^{-\frac{x}{t}}f(x)x^{\lambda-1}\td x}.
\end{align*}
The density $\omega_\lambda(z)$ has the form
\begin{align*}
 \omega_\lambda(z) &= 2\widetilde{K}_{\lambda-1}(|z|), & z\in\CC^\times,
\end{align*}
and hence the space $\calF_{\lambda,t}(\Xi)=\calF_{\lambda,t}(\CC^\times)$ is given by all entire functions $F$ on $\CC$ such that
\begin{align*}
 \|F\|^2 &= 4(2t)^{-\lambda}\int_\CC{|F(z)|^2e^{\frac{x}{t}}\widetilde{K}_{\lambda-1}\left(\frac{|z|}{t}\right)|z|^{2\lambda-1}\td z} < \infty.
\end{align*}

Consider the squaring map $\pi:\RR\setminus\{0\}\to\RR_+,\,y\mapsto y^2$ which is a two-fold cover of $\RR_+$ and induces a unitary (up to a scalar) isomorphism
\begin{align*}
 \pi^*:L^2(\RR_+,x^{\lambda-1}\td x)\to L^2_{\even}(\RR,y^{2\lambda-1}\td y),\,\pi^*F(y)=F(y^2).
\end{align*}
The differential operator on $L^2_{\even}(\RR,y^{2\lambda-1}\td y)$ corresponding to the Bessel operator is given by
\begin{align*}
 \pi^*B_\lambda(\pi^*)^{-1} &= \frac{1}{4}\left(\frac{\td^2}{\td y^2}+\frac{2\lambda-1}{y}\frac{\td}{\td y}\right).
\end{align*}
For $\lambda=\frac{1}{2}$ this is the usual Laplacian on $\RR$ and since
\begin{align*}
 \widetilde{I}_{-\frac{1}{2}}(z) &= \frac{1}{\sqrt{\pi}}\cosh z, & \widetilde{K}_{-\frac{1}{2}}(z) &= \frac{\sqrt{\pi}}{2}e^{-z},
\end{align*}
the associated heat kernel $h_{\frac{1}{2}}(t,x,y)$ and the density $\omega_{\frac{1}{2}}(z)$ transform into
\begin{align*}
 h_{\frac{1}{2}}(t,\pi(x),\pi(y)) &= \frac{1}{\sqrt{2\pi t}}e^{-\frac{x^2+y^2}{t}}\cosh\left(\frac{2xy}{t}\right), & t>0,x,y\in\RR\setminus\{0\},\\
 \omega_{\frac{1}{2}}(\pi(z)) &= \sqrt{\pi}e^{-|z|^2}, & z\in\CC\setminus\{0\}.
\end{align*}
Note that since $e^z=\cosh(z)+\sinh(z)$ the heat kernel $h_{\frac{1}{2}}(t,\pi(x),\pi(y))$ is (up to scaling) the even part of the classical heat kernel
\begin{align*}
 \frac{1}{\sqrt{2\pi t}}e^{-\frac{(x-y)^2}{t}},
\end{align*}
substituted $t$ with $4t$. We therefore recover the corresponding results for the classical heat kernel transform on even functions on $\RR$ as a special case of our theory.

\bibliographystyle{amsplain}
\bibliography{bibdb}

\vspace{30pt}

\textsc{Jan M\"ollers\\Institut for Matematiske Fag, Aarhus Universitet, Ny Munkegade 118, 8000 Aarhus C, Danmark.}\\
\textit{E-mail address:} \texttt{moellers@imf.au.dk}

%\newpage

%\addcontentsline{toc}{section}{Index}
%\printindex

\end{document}